\documentclass[10pt]{article}
\usepackage{amsmath,amsthm,amsfonts,amssymb,url}
\usepackage{authblk,longtable}
\usepackage{color}
\usepackage[T1]{fontenc}

\newtheorem{theorem}{Theorem}[section]
\newtheorem{lemma}[theorem]{Lemma}
\newtheorem{corollary}[theorem]{Corollary}
\newtheorem{remark}{Remark}
\newtheorem{example}{Example}[section]

\newcommand{\zed}{\ensuremath{\mathbb{Z}}} 

\title{New Results on Modular Golomb Rulers, Optical Orthogonal Codes and Related Structures}
\author[1]{Marco Buratti\thanks{This work has been performed under the auspices
of the G.N.S.A.G.A.\ of the C.N.R.\ (National Research
Council) of Italy}}
\author[2]{Douglas R.\ Stinson\thanks{D.R.\ Stinson's research is supported by  NSERC discovery grant RGPIN-03882.
}}
\affil[1]{Dipartimento di Matematica e Informatica, 
Universit\`{a} di Perugia, 06123, Perugia, Italy}
\affil[2]{David R.\ Cheriton School of Computer Science, University of Waterloo,
Waterloo, Ontario, N2L 3G1, Canada}

\date{\today}

\begin{document}
\maketitle

\begin{abstract}
We prove new existence and nonexistence results for modular Golomb rulers in this paper. We completely determine which 
modular Golomb rulers of order $k$ exist, for all $k\leq 11$, and we present a general existence result that holds for all $k \geq 3$. We also derive new nonexistence results for infinite classes of modular Golomb rulers and related structures such as difference packings, optical orthogonal codes, cyclic Steiner systems  and relative difference families. 

\end{abstract}

\section{Introduction and Definitions}

A \emph{Golomb ruler} of \emph{order} $k$ is a set of $k$ distinct integers, say
$x_1 < x_2 < \cdots < x_k$, such that all the differences $x_j - x_i$ ($i \neq j$) are distinct. 
To avoid trivial cases, we assume $k \geq 3$.
The \emph{length} of the ruler is $x_k - x_1$. For a survey of constructions of Golomb rulers, see \cite{Drakakis}.

A \emph{$(v,k)$-modular Golomb ruler}
(or $(v,k)$-MGR) is a set of $k$ distinct integers, \[0 \leq x_1 < x_2 < \cdots < x_k \leq v-1,\] such that all the differences $x_j - x_i \bmod v$ ($i \neq j$) are distinct elements of $\zed_v$. 
We define length and order as before.  It is obvious that a modular Golomb ruler is automatically a Golomb ruler. We can assume without loss of generality that $x_1 = 0$. 

Known results on modular Golomb rulers are summarized in \cite[\S VI.19.3]{HCD}. We state a few basic results and standard constructions now.

\begin{theorem}
If there exists a $(v,k)$-MGR, then $v \geq k^2 -k+1$. Further, a $(k^2-k+1,k)$-MGR is equivalent to a cyclic
$(k^2-k+1,k,1)$-difference set.
\end{theorem}

Of course a $(q^2+q+1,q+1,1)$-difference set (i.e., a Singer difference set) is known to exist if $q$ is a prime power. So we have the following Corollary.
\begin{corollary}
\label{Singer.cor}
There exists a $(k^2-k+1,k,1)$-MGR if $k-1$ is a prime power.
\end{corollary}

It is widely conjectured that a $(q^2+q+1,q+1,1)$-difference set exists only if $q$ is a prime power, and this conjecture has been verified for all $q < 2000000$; see \cite{Gordon}.

\begin{theorem}[Bose]
\label{bose.thm}
\cite{bose}
For any prime power $q$, there is a $(q^2-1,q)$-MGR.
\end{theorem}

\begin{theorem}[Rusza]
\label{ruzsa.thm}
\cite{ruzsa}
For any prime  $p$, there is a $(p^2-p,p-1)$-MGR.
\end{theorem}


A \emph{$(v,k;n)$-difference packing} is a set of $n$ $k$-element subsets of $\zed_v$, say $X_1, \dots , X_n$,
such that all the differences in the multiset  
\[ \{ x-y: x,y \in X_i , x \neq y , 1 \leq i \leq n\} \]
are nonzero and distinct. The following result is obvious.

\begin{theorem}
A $(v,k)$-MGR is equivalent to a $(v,k;1)$-difference packing.
\end{theorem}

A \emph{$(v,b,r,k)$-configuration} is a set system  $(V,\mathcal{B})$, where 
$V$ is a set of $v$ \emph{points} and $\mathcal{B}$ is a set of $b$ \emph{blocks}, each of which contains exactly  $k$ points, such that the following properties hold:
\begin{enumerate}
\item no pair of points occurs in more than one block, and
\item every point occurs in exactly $r$ blocks.
\end{enumerate}
It is easy to see that the parameters of a $(v,b,r,k)$-configuration satisfy the equation
$bk = vr$. For basic results on configurations, see \cite[\S VI.7]{HCD}.
A $(v,b,r,k)$-configuration is \emph{symmetric} if $v = b$, which of course implies $r=k$. In this case we speak of it as 
a \emph{symmetric $(v,k)$-configuration}.
A symmetric $(v,k)$-configuration is \emph{cyclic} if there is a cyclic permutation of the $v$ points that maps every block to a block.

We state the following easy result without proof.

\begin{theorem}
A $(v,k)$-MGR is equivalent to a cyclic symmetric $(v,k)$-configuration.
\end{theorem}

For additional connections between Golomb rulers and symmetric configurations, see \cite{BS,DFGMP}.

A \emph{$(v,k,\lambda_a,\lambda_c)$-optical orthogonal code of size $n$} is a set $\mathcal{C}$ of $n$ $(0,1)$-vectors of length $v$, which satisfies the following properties:
\begin{enumerate}
\item the hamming weight of $\mathbf{x}$ is equal to $k$, for all $\mathbf{x} \in \mathcal{C}$,
\item \emph{autocorrelation:}
for all $\mathbf{x}  = (x_0, \dots , x_{v-1}) \in \mathcal{C}$, the following holds for all integers $\tau$ such that
$0 < \tau < v$:
\[ 
\sum _{i = 0}^{v-1} x_ix_{i+\tau} \leq \lambda_a,\]
where subscripts are reduced modulo $v$.
\item \emph{cross-correlation:}
for all $\mathbf{x}  = (x_0, \dots , x_{v-1}) \in \mathcal{C}$ and all $\mathbf{y}  = (y_0, \dots , y_{v-1}) \in \mathcal{C}$ with $\mathbf{x} \neq \mathbf{y}$, the following holds for all integers $\tau$ such that
$0 \leq \tau < v$:
\[ 
\sum _{i = 0}^{v-1} x_i y_{i+\tau} \leq \lambda_c,\]
where subscripts are reduced modulo $v$.
\end{enumerate}
We sometimes abbreviate the phrase ``optical orthogonal code'' to ``OOC.''
If $\lambda_a = \lambda_c = \lambda$, then the optical orthogonal code is denoted 
as a $(v,k,\lambda)$-optical orthogonal code.

Optical orthogonal codes were introduced  by Chung,  Salehi and Wei \cite{chung} in 1989 and have been 
studied by numerous authors since then. The following result establishes the equivalence of OOC and difference packings.

\begin{theorem}
\cite{chung}
A $(v,k;n)$-difference packing is equivalent to a $(v,k,1)$-optical orthogonal code of size $n$.
\end{theorem}

The following result is proven in \cite{chung} by a simple counting argument.

\begin{theorem}
\label{optimalbound.thm}
If there exists a $(v,k,1)$-optical orthogonal code of size $n$, then
\[ n \leq \left\lfloor \frac{v-1}{k(k-1)} \right\rfloor .\]
\end{theorem}
A $(v,k,1)$-optical orthogonal code is \emph{optimal} if the relevant inequality in Theorem \ref{optimalbound.thm} is met with equality.

\bigskip

Relative difference families have been 
introduced in \cite{Buratti98} as a natural generalization of relative difference sets. We define them now.
Let $H$ be a subgroup of a finite additive group $G$, and let $k$, $\lambda$ be positive integers.
A \emph{$(G,H,k,\lambda)$-relative difference family}, or $(G,H,k,\lambda)$-RDF for short, 
is a collection $X$ of $k$-subsets of $G$ (called \emph{base blocks}) whose list
of differences has no element in $H$ and covers all elements of $G\setminus H$ exactly $\lambda$ times.
If $G$ has order $v$ and $H$ has order $w$, we say that $X$ is a $(v,w,k,\lambda)$-RDF in $G$
\emph{relative to} $H$. 
If $X$ consists of  $n$ base blocks, it is evident that 
\begin{equation}\label{BaseBlocks}
\lambda(v-w)=k(k-1)n.
\end{equation}

When $H=\{0\}$ (or, equivalently, if  $w = 1$), one usually speaks of an \emph{ordinary} $(v,k,\lambda)$-difference family
or $(v,k,\lambda)$-difference family ($(v,k,\lambda)$-DF, for short), in $G$.
If $n=1$, then we refer to a $(G,H,k,\lambda)$-relative difference family as a \emph{$(G,H,k,\lambda)$-relative difference set}.
Analogously, a $(v,k,\lambda)$-difference family of size $n=1$ is a \emph{$(v,k,\lambda)$-difference set} ($(v,k,\lambda)$-DS, for short).

\subsection{Number-theoretic Background}

In this section, we record some number-theoretic results that we will be using later in the paper.

\begin{theorem}
\label{sumofsquares.thm} \quad \vspace{-.18in}\\
\begin{enumerate} 
\item A positive integer can be written as a sum of two squares if and only if its prime decomposition contains no prime $p \equiv 3 \bmod 4$ raised to an odd power.
\item A positive integer can be written as a sum of three  squares if 
and only if it is not of the form $4^a (8b+7)$, where $a$ and $b$ are nonnegative integers.
\item Any positive integer can be written as a sum of four squares.
\end{enumerate}
\end{theorem}

\begin{proof}
Statement 1.\ is proven in many textbook on elementary number theory, e.g., \cite[Theorem 13.6]{Rosen}. 
The result 2.\ is known as
Legendre's Three-square Theorem  (for a proof of it, see, e.g., \cite[Chapter 20, Theorem 1]{Mordell}).
Finally, 3.\ is Lagrange's Four-square Theorem.
\end{proof}

\begin{lemma}
\label{sequence.lem}
For any positive integer $t$, there exist $t$ consecutive positive integers, none of which is a sum of two squares.
\end{lemma}

\begin{proof}
Take $t$ distinct primes $p_1, \dots , p_t$ all of which are  $\equiv 3
\bmod 4$ (they exist by the Dirichlet's Theorem on primes in an arithmetic progression).
By the Chinese Remainder Theorem, the system of
$t$ congruences 
\[x + i \equiv p_i \bmod {p_i}^2 
\quad \quad (1 \leq i \leq t) \] has a solution $s$.
 
Since
$s+i \equiv p_i \bmod {p_i}^2$, it is clear that $s+i$ is divisible
by $p_i$, but not by ${p_i}^2$. Since $p_i \equiv 3
\bmod 4$, it follows from Theorem \ref{sumofsquares.thm} that $s+i$ is not a sum of
two squares. This holds for $1 \leq i \leq t$.
\end{proof}

\begin{lemma}
\label{consec.lem}
Two consecutive integers, say $n$ and $n+1$, are both not expressible as a sum of three squares if and only if
$n = 4^a (8b+7) -1$, where $a \geq 2$ and $b \geq 0$.
\end{lemma}

\begin{proof}
This is a consequence of Legendre's Three-square Theorem (Theorem \ref{sumofsquares.thm}).
If $n$ is not expressible as a sum of three squares, then $n \equiv 0,4 \text{ or } 7 \bmod 8$. Therefore, if $n$ and $n+1$ 
are both not expressible as a sum of three squares, then $n \equiv  7 \bmod 8$. 
It follows from Legendre's Three-square Theorem that $n$ and $n+1$ 
are both not expressible as a sum of three squares if and only if $n + 1 = 4^a (8b+7)$ where $a\geq 2$ and $b \geq 0$.
\end{proof}

\subsection{Our Contributions}

Section \ref{MGRexistence.sec} gives existence results for modular Golomb rulers. We summarize exhaustive searches that we have carried out for all $k\leq 11$, and we present a general existence result that holds for all $k \geq 3$. 
Section \ref{nonexistenceMGR.sec} proves nonexistence results for various infinite classes of modular Golomb rulers.
Many of our new results are based on counting even and odd differences and then applying some classical results from number theory which establish which integers can be expressed as a sum of a two or three squares. 
Section \ref{nonexistenceOOC.sec} studies optical orthogonal codes and provides nonexistence results for certain optimal OOCs.
In Section \ref{other.sec}, we consider cyclic Steiner systems and relative difference families and we present additional nonexistence results using the techniques we have developed. Finally, Section \ref{summary.sec} is a brief summary.


\section{Existence Results for $(v,k)$-MGR}
\label{MGRexistence.sec}

In this section, we report the results of exhaustive searches for $(v,k)$-MGR with $k \leq 11$. 
We also prove a general existence result that holds for all integers $k \geq 3$. First, we discuss a few preliminary results..

Given a positive integer $k \geq 3$, define
\[ \mathsf{MGR}(k) = \{ v : \text{there exists a $(v,k)$-MGR} \} .\] 
We are interested in the set $\mathsf{MGR}(k)$. 
In particular, it is natural to try to determine the minimum integer in $\mathsf{MGR}(k)$ 
as well as the maximum integer not in $\mathsf{MGR}(k)$.

Another parameter of interest is the length of a  Golomb ruler. There has been considerable research done on finding the minimum length of a Golomb ruler of specified order $k$, which we denote by $L^*(k)$. In the modular case,
we will define $L_m^*(k)$ to be the minimum  $L$ such that there exists a $(v,k)$-MGR of length $L$ for some $v$.

The following basic lemma is well-known.

\begin{lemma} 
\label{double.lem}
Suppose there is a Golomb ruler of order $k$ and length $L$, and suppose $v \geq 2L+1$. Then there
is a $(v,k)$-MGR.
\end{lemma}

\begin{proof}
We have a Golomb ruler consisting of $k$ integers $0 = x_1 < x_2 < \cdots < x_k = L$.
Consider these as residues modulo $v$, where $v \geq 2L+1$. Clearly all the ``positive residues'' 
$x_j - x_i \bmod v$ ($i < j$) are nonzero and distinct, as are all the ``negative residues'' $x_j - x_i \bmod v$ ($j < i$).
The largest positive residue is $L$ and the smallest negative residue is
$v - L$. Since $v >2 L$, no positive residue is equal to a negative residue.
\end{proof}

The following is an immediate consequence of Lemma \ref{double.lem}. 

\begin{lemma}
For any positive integer $k \geq 2$, $L^*(k) = L_m^*(k)$.
\end{lemma}

Given a positive integer $k \geq 3$, define
\[ \mathsf{MGR}(k) = \{ v : \text{there exists a $(v,k)$-MGR} \} .\]
We have performed exhaustive backtracking searches in order to determine the sets $\mathsf{MGR}(k)$ for 
$3 \leq k \leq 11$. For each value of $k$, once we have constructed a sufficient number of ``small'' $(v,k)$-MGR,
we can apply Lemma \ref{double.lem} to conclude that all $(v,k)$-MGR exist for larger values of $v$. To this end, when we compute all the $(v,k)$-MGR for given values of $v$ and $k$, we keep track of the ruler having the smallest possible length.
This facilitates the application of Lemma \ref{double.lem}

Our computational results are summarized as follows.

\begin{theorem}
\begin{enumerate}
\item $\mathsf{MGR}(3) = \{v : v\geq 7\}$.
\item $\mathsf{MGR}(4) = \{v : v\geq 13\}$.
\item $\mathsf{MGR}(5) = \{21\} \cup \{v : v\geq 23\}$.
\item $\mathsf{MGR}(6) = \{31\} \cup \{v : v\geq 35\}$.
\item $\mathsf{MGR}(7) = \{v : v\geq 48\}$.
\item $\mathsf{MGR}(8) = \{57\} \cup \{v : v\geq 63\}$.
\item $\mathsf{MGR}(9) = \{73,80\} \cup \{v : v\geq 85\}$.
\item $\mathsf{MGR}(10) = \{91\} \cup \{v : v\geq 107\}$.
\item $\mathsf{MGR}(11) = \{120,133\} \cup \{v : v\geq 135\}$.
\end{enumerate}
\end{theorem}

\begin{proof}
Proof details are in  Table \ref{MGR.tab}.
\end{proof}

\begin{longtable}{c|r|l}
\caption{$(v,k)$-modular Golomb rulers for $3 \leq k \leq 11$}\\
$v$ & $k$ & \multicolumn{1}{|c}{ruler} \\ \hline
\endfirsthead
\caption{$(v,k)$-modular Golomb rulers for $3 \leq k \leq 11$ (cont.)}\\
$v$ & $k$ & \multicolumn{1}{|c}{ruler} \\ \hline
\endhead
\label{MGR.tab}
$v = 7$ & $3$ & $0 ,1 ,3$  \\ 
$v\geq 8$  & $3$ &  {Lemma \ref{double.lem}}, $v=7$, $L = 3$  \\ \hline 
$v = 13$ & $4$	&	$0 , 1 , 4 , 6$		\\ 
$v\geq 14$ & $4$ &  {Lemma \ref{double.lem}}, $v=13$, $L = 6$ 	 \\ \hline 
$v = 21$ & $5$	&	 $0 ,  2 ,  7 ,  8 , 11$		\\
$v = 22$ & $5$ & {does not exist} \\
$v\geq 23$ & $5$ & {Lemma \ref{double.lem}}, $v=21$, $L = 11$ \\ \hline 
$v = 31$ & $6$  &	  $0  , 1  , 4 , 10 , 12 , 17$ 		\\ 
$32 \leq v \leq 34$ & $6$ & {does not exist} \\
$v \geq 35$ & $6$ & {Lemma \ref{double.lem}}, $v=31$, $L = 17$ \\ \hline 
$43 \leq v \leq 47$ & $7$ & {does not exist} \\
$v = 48$ & $7$	&	$0  , 5  , 7 , 18 , 19 , 22  ,28$	\\
$v = 49$ & $7$	&	$0  , 2  , 3 , 10 , 16 , 21  ,25$ 	\\
$v = 50$ & $7$	&	$0  , 1  , 5 ,  7 , 15 , 18  ,27$ \\ 
$v \geq 51$ & $7$ & {Lemma \ref{double.lem}}, $v=49$, $L = 25$ \\ \hline 
$v = 57,64,68$ & $8$	&	$0  , 4  , 5 , 17 , 19 , 25 , 28 , 35$ 	\\
$58 \leq v \leq 62$ & $8$ & {does not exist} \\
$v = 63,67$ & $8$	&	$0  , 1  , 8 , 20 , 22 , 25 , 31 , 35$\\
$v = 65$ & $8$	&	$0 ,  2  ,10 , 11 , 16 , 28 , 31 , 35$\\
$v = 66$ & $8$	&	$0  , 2 , 10 , 21 , 24 , 25 , 30 , 37$\\
$v = 69$ & $8$	&	$0  , 1  , 4 ,  9 , 15 , 22 , 32 , 34$ \\
$v \geq 70$ & $8$ & {Lemma \ref{double.lem}}, $v=69$, $L = 34$  \\ \hline 
$v = 73$ & $9$	&	$0 ,  2 , 10 , 24 , 25 , 29 , 36 , 42 , 45$\\
$74 \leq v \leq 79$ & $9$	&	{does not exist} \\
$v = 80$ & $9$	&	$0  , 1 , 12 , 16 , 18 , 25 , 39 , 44 , 47$\\
$81 \leq v \leq 84$ 	& $9$	&	{does not exist} \\
$v = 85$ & $9$	&	$0 ,  1 ,  7  ,12 , 21  ,29 , 31 , 44 , 47$\\
$v = 86,88$ & $9$	&	$0 ,  2 ,  5  ,13 , 17 , 31 , 37 , 38 , 47$\\
$v = 87$ & $9$	&	$0 ,  1  , 4  ,13  ,24 , 30 , 38 , 40 , 45$\\
$v = 89$ & $9$ 	&   $0  , 1 ,  5 , 12 , 25  ,27 , 35 , 41 , 44$\\
$v \geq 90$ & $9$ & {Lemma \ref{double.lem}}, $v=89$, $L = 44$ \\ \hline
$v = 91$ & $10$	& $0  , 1 ,  6 , 10 , 23 , 26 , 34 , 41 , 53 , 55$	\\
$92 \leq v \leq 106$ & 10	&	{does not exist} \\
$v = 107$ & $10$	&  $0  , 2  ,15 , 21 , 22 , 32 , 46 , 50 , 55 , 58$\\
$v = 108$ & $10$	&  	 $0  , 2  , 8 , 27 , 32 , 36 , 39 , 49,  50 , 65$\\
$v = 109$ & $10$	&  $0  , 4 , 11 , 16  ,25 , 35 , 38 , 53 , 55 , 61$	\\
$v = 110$ & $10$	&  $0  , 3  ,14 , 16 , 36 , 37 , 42 , 46 , 54 , 61$	\\
$v \geq 111$ & $10$ & {Lemma \ref{double.lem}}, $v=91$, $L = 55$  \\\hline 
$111 \leq v  \leq 119$ & $11$	&	{does not exist} \\
$ v = 120$ & $11$	&   $0  , 1  , 4   ,9,  23 , 30 , 41 , 43 , 58 , 68 , 74$\\
$121 \leq v \leq 132$ & $11$	&	{does not exist} \\
$ v = 133$ & $11$	&  $0 ,  1  , 9  ,19 , 24 , 31  ,52 , 56 , 58 , 69 , 72$ \\
$v = 134 $ & $11$	&	{does not exist} \\
$ v = 135$ & $11$	&  $0  , 5  , 7  ,11 , 31 , 41 , 49 , 50 , 63 , 66,  78$ \\
$ v = 136$ & $11$	&  $0 ,  2 , 11 , 27 , 37 , 42 , 45 , 59 , 65 , 66 ,78$ \\
$ v = 137$ & $11$	&  $0  , 1  ,16  ,21 , 24 , 33 , 43 , 61 , 68 , 72 , 74$ \\
$ v = 138$ & $11$	&  $0  , 4 ,  5 , 23 , 25 , 37 , 52 , 59 , 65 , 68 , 76 $\\
$ v = 139$ & $11$	&  $0 ,  1 ,  3  ,11  ,25  ,41 , 45 , 54 , 60 , 72 , 77$ \\
$ v = 140$ & $11$	&  $0  , 4 , 10  ,24 , 25 , 27 , 36  ,43 , 65 , 73 , 78$ \\
$ v = 141$ & $11$	&  $0  , 2 ,  3  , 7 , 20 , 29 , 41 , 52 , 60 , 66 , 76$\\
$ v = 142$ & $11$	&   $0  , 1 , 13 , 16 , 22 , 33 , 47 , 51 , 70 , 75 , 77$\\
$ v = 143,144$ & $11$	&   $0  , 3  , 7 , 22  ,27 , 43 , 56 , 57 , 66 , 68  ,74$\\
$v \geq 145$ & $11$ & {Lemma \ref{double.lem}}, $v=133$, $L = 72$ \\ \hline 
\end{longtable}

\begin{remark}
Existence of a $(110,10)$-MGR also follows from Theorem \ref{ruzsa.thm}, and 
existence of a $(48,7)$-MGR and a $(120,11)$-MGR  follow  from Theorem \ref{bose.thm}.
\end{remark}

The rulers that are presented in Table \ref{MGR.tab} provide upper bounds on $L^*_m(k)$ for $3 \leq k \leq 11$.
However, it turns out that all these values are in fact exact. This is because the exact values of $L^*(k)$ are known for small $k$ (see, for example, \cite[Table 2.2]{Dim}) and they match the minimum lengths of the modular Golomb rulers that we have recorded in Table \ref{MGR.tab}.  Thus we have the following result.

\begin{theorem}
$L^*_m(3) = 3$; $L^*_m(4) = 6$; $L^*_m(5) = 11$; $L^*_m(6) = 17$; $L^*_m(7) = 25$; $L^*_m(8) = 34$; 
$L^*_m(9) = 44$; $L^*_m(10) = 55$; and $L^*_m(11) = 72$.
\end{theorem}


Now we state and prove two general existence results that hold for all $k \geq 3$. 

\begin{theorem}
\label{exist1.thm}
For any integer $k \geq 3$, there is a $(v,k)$-MGR for some integer $v \leq 3k^2/2$.
\end{theorem}

\begin{proof}
For $3 \leq k\leq 11$, we refer to the results in Table \ref{MGR.tab}. Indeed, for these values of $k$, there is a 
$(v,k)$-MGR for some integer $v \leq k^2-1$.

For $12 \leq k \leq 24$, we use Corollary \ref{Singer.cor}. 
There is a $(p^2+p+1,p+1,1)$-difference set in $\zed_{p^2+p+1}$ for $p = 11,13,16,17,19$ and $23$.
If we delete $\delta = p+1 - k$ elements from such a difference set, we obtain
a $(p^2+p+1,k)$-MGR. For $k =12$, we have $p= 11$ and $\delta = 0$;
for $k =13,14$, we have $p= 13$ and $\delta \leq 1$; 
for $15 \leq k \leq 17$, we have $p= 16$ and $\delta \leq 2$; 
for $k =18$, we have $p= 17$ and $\delta = 0$; 
for $k =19,20$, we have $p= 19$ and $\delta \leq 1$; and 
for $21 \leq k \leq 24$, we have $p= 23$ and $\delta \leq 3$.
So, for $12 \leq k \leq 24$, there is a $(v,k)$-MGR for some integer 
\begin{align*} v &\leq (k + \delta - 1)^2 + (k + \delta - 1) + 1 \\
&\leq (k+2)^2 +k + 3 \\
&= k^2 +5k + 7.
\end{align*}
It is easy to verify that $k^2 +5k + 7 \leq  3k^2/2$ if $k \geq 12$.

Finally, suppose $k \geq 25$. Let $p$ be the smallest prime such that $p \geq k - 1$.
By a result of Nagura \cite{Nagura}, we have $p \leq 6(k-1)/5 < 6k/5$.
From Corollary \ref{Singer.cor}, there exists a $(p^2+p+1,p+1,1)$-difference set in $\zed_{p^2+p+1}$. Delete 
$p+1 - k$ elements from this difference set to obtain
a $(p^2+p+1,k)$-MGR. 
We have 
\begin{align*}
p^2+p+1 &< \left( \frac{6k}{5} \right)^2 + \frac{6k}{5} + 1\\
&< \frac{3k^2}{2},  
\end{align*}
where the last inequality holds for $k \geq 21$.
\end{proof}

\begin{theorem}
\label{exist2.thm}
For any integer $k \geq 3$ and any integer $v \geq 3k^2 - 1$, there is a $(v,k)$-MGR.
\end{theorem}

\begin{proof}
From Theorem \ref{exist1.thm}, there  exists a $(v,k)$-MGR for some integer $v \leq 3k^2/2$.
This ruler has length $L \leq  3k^2/2 - 1$.  Applying Theorem \ref{double.lem},
there
is a $(v,k)$-MGR for all $v \geq 2(3k^2/2 - 1)+1 = 3k^2 - 1$. 
\end{proof}

\begin{remark} Of course there are stronger results known on gaps between consecutive primes that hold for larger integers.
For example,  it was shown by Dusart \cite{Dusart} that, 
if $k \geq 89693$, then there is at least one prime $p$ such that 
\[ k < p \leq \left( 1 + \frac{1}{\ln^3 k} \right) k .\] So improved versions of Theorems \ref{exist1.thm} and \ref{exist2.thm}
could be proven that hold for sufficiently large values of $k$.
\end{remark}

\section{Nonexistence Results for $(v,k)$-MGR}
\label{nonexistenceMGR.sec}

We present several nonexistence results for infinite classes of modular Golomb rulers in this section.

\subsection{$(k^2-k+2,k)$-MGR}

We have noted that $v \geq k^2-k+1$ if a $(v,k)$-MGR exists, and the  $(k^2-k+1,k)$-MGR are equivalent to 
cyclic difference sets with $\lambda = 1$. 
There has been considerable study of these difference sets and various nonexistence results are known.
We do not discuss this case further here, but we refer to \cite[\S 8]{Jung-survey} for a good summary of known results.

The next case is $v =  k^2-k+2$.
First, we note that there are two small examples of $(k^2-k+2,k)$-MGR, namely, an $(8,3)$-MGR and a $(14,4)$-MGR.
These are found in Table \ref{MGR.tab}. In fact, these are the only examples that are known to exist. We now discuss some nonexistence results for $(k^2-k+2,k)$-MGR.

We next observe that $(k^2-k+2,k)$-MGR are equivalent to certain relative difference sets
in the cyclic group $\zed_{k^2-k+2}$. The proof of this easy result is left to the reader.


\begin{theorem}
\label{equiv.thm}
A $(k^2-k+2,k)$-MGR is equivalent to a $(\zed_{k^2-k+2}, H, k, 1)$-relative difference set, 
where $H$ is the unique subgroup of order $2$ in $\zed_{k^2-k+2}$, i.e., $H = \{0,(k^2-k+2)/2\}$.
\end{theorem}


It is well-known that relative difference sets give rise to certain square divisible designs, which we define now.
A \emph{$(w,u,k,\lambda_1,\lambda_2)$-divisible design} is a set system (actually, a type of group-divisible design) on $v = uw$ points and having blocks of size $k$, such that the following conditions are satisfied:
\begin{enumerate}
\item
the points are partitioned into $u$ groups of size $w$,
\item two points in the same group occur together in exactly $\lambda_1$ blocks, and
\item two points in different groups occur together in exactly $\lambda_2$ blocks.
\end{enumerate}
If the number of blocks is the same as the number of points, then we have a \emph{square divisible design}.

The following result is a consequence of Theorem \ref{equiv.thm}, since a square divisible design is obtained by developing a relative difference set through the relevant cyclic group.

\begin{theorem}
\label{divisible.thm}
If there exists a  $(k^2-k+2,k)$-MGR, then there exists a square divisible design with parameters 
$w=2$, $u = (k^2 - k +2)/2$, $\lambda_1 = 0$ and $\lambda_2 = 1$.
\end{theorem}

We will make use of some results due to Bose and Connor \cite{BC}, as stated in \cite[Proposition 1.8]{Jung}.

\begin{theorem}[Bose and Connor]
\label{BC.thm}
Suppose there exists a square divisible design with parameters $w$, $u$, $k$, $\lambda_1$ and $\lambda_2$.
Denote $v = uw$.
Then the following hold.
\begin{enumerate}
\item If $u$ is even, then $k^2 - \lambda_2 v$ is a perfect square. If furthermore $u \equiv 2 \bmod 4$, then
$k - \lambda_1$ is the sum of two squares.
\item If $u$ is odd and $w$ is even, then $k - \lambda_1$ is a perfect square and the equation 
\[ (k^2 -  \lambda_2v)x^2 + (-1)^{u(u-1)/2}  \lambda_2 w y^2 = z^2 \] has a nontrivial solution in integers $x$, $y$ and $z$.
\end{enumerate}
\end{theorem}

We can use Theorem \ref{BC.thm} to obtain necessary conditions for the existence of
$(k^2-k+2,k)$-MGR.

\begin{corollary}
\label{BC.cor}
Suppose there exists a $(k^2-k+2,k)$-MGR.
Then the following hold.
\begin{enumerate}
\item $k \not\equiv 7 \bmod 8$.
\item If $k \equiv 2\bmod 8$, then $k-2$ is a perfect square and $k$ is the sum of two squares.
\item If $k \equiv 3, 6\bmod 8$, then $k-2$ is a perfect square. 
\item If $k \equiv 0,1\bmod 8$, then $k$ is a perfect square and the equation 
\[ (k-2)x^2 + 2y^2 = z^2 \] has a nontrivial solution in integers $x$, $y$ and $z$.
\item If $k \equiv 4,5\bmod 8$, then $k$ is a perfect square and the equation 
\[ (k-2)x^2 - 2y^2 = z^2 \] has a nontrivial solution in integers $x$, $y$ and $z$.
\end{enumerate}
\end{corollary}

\begin{proof}
Suppose there exists a $(k^2-k+2,k)$-MGR.
Then, from Theorem \ref{divisible.thm}, there is a square divisible design with parameters $w=2$, $u = (k^2 - k +2)/2$, $v = k^2 - k +2$, $\lambda_1 = 0$ and $\lambda_2 = 1$.  We apply Theorem \ref{BC.thm}, making use of the fact that $k^2- \lambda_2 v = k-2$.

First, we observe that $u$ is even if and only if $k \equiv 2,3 \bmod 4$.  Further, $u \equiv 2 \bmod 4$ if and only if $k \equiv 2, 7 \bmod 8$.

If $k \equiv 7 \bmod 8$, then $k^2- \lambda_2 v = k-2 \equiv 5 \bmod 8$, so $k^2- \lambda_2 v$ is not a perfect square. Therefore, from part 1.\ of Theorem \ref{BC.thm}, a 
$(k^2-k+2,k)$-MGR does not exist if $k \equiv 7 \bmod 8$.

If $k \equiv 2 \bmod 8$, then part 1.\ of Theorem \ref{BC.thm} says that $k-2$ is a perfect square and $k$ is the sum of two squares.

If $k \equiv 3,6 \bmod 8$, then part 1.\ of Theorem \ref{BC.thm} says that $k-2$ is a perfect square.

When $k \equiv 0,1 \bmod 8$, we have $u \equiv 1 \bmod 4$ and hence $(-1)^{u(u-1)/2} = 1$.
When  $k \equiv 4,5 \bmod 8$, we have $u \equiv 3 \bmod 4$ and hence $(-1)^{u(u-1)/2} = -1$.
The stated results then follow immediately from part 2.\ of Theorem \ref{BC.thm}.

\end{proof}

\subsection{$(k^2-k+2\ell,k)$-MGR}

For $v > k^2-k+2$, a $(v,k)$-MGR is not necessarily a relative difference set and it does not necessarily imply the existence of 
a square divisible design. So, in general, we cannot apply the results in Theorem \ref{BC.thm}.
However, we can derive some nice necessary conditions for the existence of certain $(v,k)$-MGR using elementary counting arguments. These arguments are in the spirit of  techniques introduced in \cite[\S 2]{Buratti99}; see also \cite{MDV}. 
Before studying MGR, we present a simple example to illustrate the basic idea.

\begin{example}
\label{BRC.exam}
Suppose we have 
a $(v,k,\lambda)$-difference set in $\zed_v$ when $v$ is even. There are $v/2-1$ nonzero even differences 
and $v/2$ odd differences, each of which occurs $\lambda$ times. 
Suppose the difference set consists of $a$ even elements and $b$ odd elements.
Then $a+b = k$ and $2ab = \lambda v/2$. So $a$ and $b$ are the solutions of the quadratic equation
\[ x^2 - kx + \frac{\lambda v}{4} =0.\]
Since $a$ and $b$ are integers, the discriminant must be a perfect square. Therefore, 
$k^2 - \lambda v$ is a square. 
However, $k(k-1) = \lambda(v-1)$, so $k^2 - \lambda v = k - \lambda$ 
must be a perfect square. (Of course, this condition is the same as in the Bruck-Ryser-Chowla Theorem 
for $v$ even, which holds for any symmetric BIBD.)
\end{example}

%
%
%
In the next theorem, we will use this counting technique to obtain necessary conditions for the existence of a $(k^2-k+2\ell,k)$-MGR for a given  integer $\ell \geq 1$. First, we give a couple of definitions that will be useful in the rest of the paper.

Suppose $X$ is a $(v,k)$-MGR. Define \[ \Delta X = \{ x - y \bmod v: x,y \in X, x \neq y\} \]
and \[L(X) = \zed_v \setminus \Delta X.\] Note that $\Delta X$ consists of all the differences obtained from pairs of distinct elements in $X$ and $L(X)$ is the complement of $\Delta X$. The set $L(X)$ is called the \emph{leave} of $X$.
For $i = 0,1$, define $L_i(X)$ to consist of the elements of $L(X)$ that are congruent to $i$ modulo $2$.

The following lemma is straightforward but useful.

\begin{lemma}
\label{even-odd.lem}
Suppose $X$ is a $(v,k)$-MGR where $v$ is even. Then
$\{0,v/2\} \subseteq L(X)$. If $v \equiv 0 \bmod 4$, then
 $|L_0(X)|$ and $|L_1(X)|$ are both even. If $v \equiv 2 \bmod 4$, then
$|L_0(X)|$ and $|L_1(X)|$ are both odd.
\end{lemma}

\begin{proof}
It is evident that $0 \in L(X)$. Also, if we have
$x-y=v/2$ for some pair $(x,y)\in X\times X$, then
we have $y-x=v/2$ as well. This would imply that $v/2$
appears at least twice as a difference, which is not allowed.
Hence $\{0,v/2\}\subset L(X)$.

Now note that if $d \in \Delta X$, then $v-d\in \Delta X$
as well.
Consequently, if $d \in L(X)$, then $v-d \in L(X)$.
Of course $d = v-d$ if and only if $d = 0$ or $d = v/2$. The remaining
elements of $\zed_v$ can be matched into pairs $(d,v-d)$ having
the same parity. Thus, considering that $v/2$ is even
or odd according to whether $v\equiv 0$ or $2$ modulo $4$,
respectively, it is clear that $|L_1(X)|$ and $|L_2(X)|$
are both even in the first case and both odd in the
second.
\end{proof}

\begin{theorem}
\label{counting2.thm}
Suppose $v = k^2-k+2\ell$, where $\ell \geq 1$, and suppose there is a $(v,k)$-MGR. Then the following hold.
\begin{enumerate}
\item If $v \equiv 2 \bmod 4$, then $k - 2\ell + 2 + 4i$ 
is a perfect square for some integer $i \in \{0, \dots ,  \ell - 1\}$.
\item If $v \equiv 0 \bmod 4$, then $k - 2\ell  + 4i$ is 
a perfect square for some integer $i \in \{0, \dots ,  \ell - 1\} $.
\end{enumerate}
\end{theorem}

\begin{proof}
Let $X$  be a $(v,k)$-MGR. Since $|X| = k$, we have 
\[ |L(X)| = v  - (k^2 - k) = 2\ell .\]
Suppose $X$ contains $a$ even elements and $b$ odd elements;
then $a+b = k$. 

Suppose first that $v \equiv 2 \bmod 4$, so $v/2$ is odd.
From Lemma \ref{even-odd.lem}, 
$|L_1(X)|$ is odd, say $|L_1(X)| = 2i+1$, and $v/2 \in L_1(X)$.
Therefore, $0 \leq i \leq \ell - 1$.

The quantity $2ab$ is equal to  the number of odd differences in $\Delta X$,
so \[2ab = \frac{v}{2} - (2i+1) = \frac{v - 2 - 4i}{2}.\]
It follows that $a$ and $b$ are the solutions of the  quadratic equation
\[ x^2 - kx + \frac{v - 2 - 4i}{4} = 0.\] 
The solutions $a$ and $b$ must be integers, which can happen only if the discriminant is a perfect square. Hence,
\[ k^2 - (v - 2 - 4i) = k - 2\ell + 2 + 4i\] is a perfect square.
Hence, $k - 2\ell + 2 + 4i$ is a perfect square for some integer $i \in \{0, \dots ,  \ell - 1\} $.

The proof is similar when $v \equiv 0 \bmod 4$. 
Here, from Lemma \ref{even-odd.lem}, $|L_1(X)|$ is even, say $|L_1(X)| = 2i$ and $\{0,v/2\} \subseteq L_0(X)$. 
Since $\{0,v/2\} \subseteq L_0(X)$, we have
$|L_1(X)| \leq 2 \ell-2$. Hence $i \in \{0, \dots ,  \ell -1\} $. 

We have
\[ 2ab = \frac{v}{2} - 2i = \frac{v - 4i}{2}.\]
It follows that $a$ and $b$ are the solutions of the  quadratic equation
\[ x^2 - kx + \frac{v  - 4i}{4} = 0.\] 
The solutions $a$ and $b$ must be integers, which can happen only if the discriminant is a perfect square. Hence,
\[ k^2 - (v  - 4i) = k - 2\ell  + 4i\] is a perfect square.
Hence, $k - 2\ell  + 4i$ is a perfect square for some integer $i \in \{0, \dots ,  \ell - 1\} $.
\end{proof}

\begin{example}
Suppose $k = 10$ and $v = 94 = 10 \times 9 + 2 \times 2$, $\ell = 2$. 
Here $v\equiv 2 \bmod 4$.  Then we compute
\[10 - 2 \times 2 + 2 + 4i = 8 + 4i\] for $i = 0,1$, obtaining $8$ and $12$. Neither of these is a perfect square, so 
we conclude that a $(94,10)$-MGR does not exist.
\end{example}



It is interesting to see what Theorem \ref{counting2.thm} tells us when $\ell = 1$. 

\begin{corollary}
\label{counting.thm}
Suppose there is a $(k^2-k+2,k)$-MGR. Then the following hold.
\begin{enumerate}
\item If $k \equiv 2,3\bmod 4$, then $k-2$ is a perfect square. 
\item If $k \equiv 0,1\bmod 4$, then $k$ is a perfect square.
\end{enumerate}
\end{corollary}

\begin{proof} Take $\ell=1$ in Theorem \ref{counting2.thm}; 
then $v = k^2 -k + 2$ and we have $i = 0$. We note that $v \equiv 0 \bmod 4$ if $k \equiv 2,3\bmod 4$
and $v \equiv 2 \bmod 4$ if $k \equiv 0,1\bmod 4$, so the stated results follow immediately.
\end{proof}

\begin{remark} We observe that Theorem \ref{divisible.thm} and 
Corollary \ref{BC.cor} provide stronger necessary conditions for the existence of $(k^2-k+2,k)$-MGR
than those stated in Corollary \ref{counting.thm}. 
\end{remark}

For certain values of $k$, we are able to find ``intervals'' in which MGR cannot exist.
Define \[S_{k,\ell} = \{k - 2\ell + 2 + 4i: 0 \leq i \leq \ell - 1\}\] 
and define \[T_{k,\ell} = \{k - 2\ell  + 4i: 0 \leq i \leq \ell - 1\}.\] 

\begin{lemma}
\label{congruent.lem}
Suppose $v = k^2-k+2\ell$. 
\begin{enumerate}
\item If $v \equiv 2 \bmod 4$, then  all elements of $S_{k,\ell}$ are $\equiv 0,1 \bmod 4$. 
\item If $v \equiv 0 \bmod 4$, then all elements of $T_{k,\ell}$ are $\equiv 0,1 \bmod 4$. 
\end{enumerate}
\end{lemma}

\begin{proof}
We prove 1. Suppose $v = 2\ell + k^2 - k \equiv 2 \bmod 4$. Then
$2\ell + k^2 - k - 2 - 4i \equiv 0 \bmod 4$.
It follows that $k^2  \equiv k - 2\ell+ 2 + 4i \bmod 4$. Since $k^2 \equiv 0,1 \bmod 4$ for all integers $k$, 
the result follows. 

The proof of 2.\ is similar.
\end{proof}

\begin{theorem}
\label{main-nonexist.thm}
Let $t$ be a positive integer.
\begin{enumerate}
\item If $k = 4t^2+4t+4$, then there does not exist
a $(k^2-k+4s,k)$-MGR for all $s$ such that $1\leq s \leq t$.
\item If $k = 4t^2+4t+2$, then there does not exist
a $(k^2-k+4s,k)$-MGR for  all $s$ such that $1\leq s \leq t$.
\item If $k = 4t^2+3$, then there does not exist
a $(k^2-k+4s-2,k)$-MGR for  all $s$ such that $1\leq s \leq t$.
\item If $k = 4t^2+1$, then there does not exist
a $(k^2-k+4s-2,k)$-MGR for  all $s$ such that $1\leq s \leq t$.
\end{enumerate}
\end{theorem}

\begin{proof}
We prove 1.  Denote $\ell = 2s$, where $1\leq s \leq t$ and let $v = k^2-k+4s$.
Since $k \equiv 0 \bmod 4$, we have $v \equiv 0 \bmod 4$. 
So we examine the elements in $T_{k,\ell}$, which are all  congruent to 
$0$ modulo $4$ by Lemma \ref{congruent.lem}.
For the smallest element of  $T_{k,\ell}$, which is $k - 2\ell $, we have 
\begin{align*} k - 2\ell &\geq k - 4t \\
&= 4(t^2+t+1) - 4t \\
&= 4t^2 + 4 \\
& > (2t)^2.
\end{align*}
Similarly, for the largest element of $T_{k,\ell}$, which is $k - 2\ell + 4(\ell-1)$, we have
\begin{align*} k - 2\ell + 4(\ell-1) 
&\leq k + 4t - 4 \\
&= 4(t^2+t+1) + 4t - 4 \\
&= 4t^2 + 8t \\
& < (2t+2)^2.
\end{align*}
Since all the elements of $T_{k,\ell}$ are congruent to 
$0$ modulo $4$ and they are between two consecutive even squares, there cannot be any perfect squares in the set $T_{k,\ell}$.

The proofs of 2., 3.\ and 4.\ are similar. 
\end{proof}

\begin{example}
If we take $t=3$ in Theorem \ref{main-nonexist.thm}, we see that
there does not exist a $(k^2-k+4,k)$-MGR, a $(k^2-k+8,k)$-MGR or
a $(k^2-k+12,k)$-MGR when $k = 50, 52$. Further, there does not exist a $(k^2-k+2,k)$-MGR, a $(k^2-k+6,k)$-MGR or
a $(k^2-k+10,k)$-MGR when $k = 37,39$. 
\end{example}

We will show that we can improve Theorem \ref{counting2.thm} when $v \equiv 0 \bmod 4$.  First we state and prove a simple numerical lemma.

\begin{lemma}
\label{numerical.lem}
Let $a$ be a positive integer. Then  
\begin{equation}
\label{num1.eq}
 \left\{h(a-h) \ : \ 0\leq h \leq \left\lfloor \frac{a}{2} \right\rfloor \right\}
=
\left\{ \left( \frac{a}{2} \right)^2-\left( \frac{a}{2} - h \right)^2 \ : \ 0\leq h \leq \left\lfloor \frac{a}{2} \right\rfloor \right\}.
\end{equation}
Further, if $a$ is even, then 
\begin{equation}
\label{num2.eq} \left\{h(a-h) \ : \ 0\leq h \leq   \frac{a}{2}   \right\}
=
\left\{ \left( \frac{a}{2} \right)^2- h^2 \ : \ 0\leq h \leq   \frac{a}{2}   \right\}.
\end{equation}
\end{lemma}

\begin{proof}
Clearly we have
\[ h(a-h) = \left( \frac{a}{2} \right)^2 - \left( \frac{a}{2} - h \right)^2.\]
Therefore (\ref{num1.eq}) holds. 
If $a$ is even, then 
\[ \left\{ \left( \frac{a}{2} \right)^2-\left( \frac{a}{2} - h \right)^2 \ : \ 0\leq h \leq  \frac{a}{2}  \right\}
=
\left\{ \left( \frac{a}{2} \right)^2-h^2 \ : \ 0\leq h \leq  \frac{a}{2}  \right\}
,\]
and (\ref{num2.eq}) holds.
\end{proof}

\begin{theorem}
\label{new3.5.thm}
Suppose that $X$ is a $(v,k)$-MGR with $v=k^2-k+2\ell$. Then the following hold.
\begin{enumerate}
\item If $v\equiv 0 \bmod 8$, then there exist integers  $i\in\{0,1,\dots,\ell-1\}$
and $j\in\{0,1,\dots,\ell-1-i\}$ such that $k-2\ell+4i$ is a perfect square 
and $k-2\ell+2i+4j$ is a sum of two squares.
\item If $v\equiv 4 \bmod 8$, then there exist integers  $i\in\{0,1,\dots,\ell-1\}$
and $j\in\{0,1,\dots,\ell-1-i\}$ such that that $k-2\ell+4i$ is a perfect square and 
$k-2\ell+2i+4j+2$ is a sum of two squares.
\end{enumerate}
\end{theorem}  

\begin{proof}
Suppose  $v\equiv 0 \bmod 8$; then $v/2 \equiv 0 \bmod 4$. From Lemma \ref{even-odd.lem} and the 
proof of Theorem \ref{counting2.thm}, there are an even number, say $2i$,
of odd elements in $L(X)$,  where $0\leq i\leq \ell-1$. 
The number of elements $\equiv 2 \bmod 4$ that are in $L(X)$  is also even, say $2j$, 
and we must have $0\leq j\leq \ell-1-i$.

Let $a$ and $b$ be the number of even and odd elements in $X$, respectively.
We showed in the proof of Theorem \ref{counting2.thm} that 
$a$ and $b$ are the solutions to the  quadratic equation \[x^2-kx+\frac{v}{4}-i=0,\] and 
hence $a+b = k$, $ab = \frac{v}{4}-i$, and
$k^2-v+4i=k-2\ell+4i$ is a perfect square. 

Let $n_\alpha$ be the number of elements of $X$ that are congruent to $\alpha$ modulo $4$, for $\alpha=0,1,2,3$.
It is evident that $n_0+n_2=a$ and that $n_1+n_3=b$. Thus, from (\ref{num1.eq}) in Lemma \ref{numerical.lem}, we have
\[
n_0n_2\in
\left\{h(a-h) \ : \ 0\leq h \leq \left\lfloor \frac{a}{2} \right\rfloor \right\}
=
\left\{ \left( \frac{a}{2} \right)^2-\left( \frac{a}{2} - h \right)^2 \ : \ 0\leq h \leq \left\lfloor \frac{a}{2} \right\rfloor \right\}
\]
and
\[
n_1n_3\in
 \left\{h(b-h) \ : \ 0\leq h \leq \left\lfloor \frac{b}{2} \right\rfloor \right\}
=
\left\{ \left( \frac{b}{2} \right)^2-\left( \frac{b}{2} - h \right)^2 \ : \ 0\leq h \leq \left\lfloor \frac{b}{2} \right\rfloor \right\}.
\]

\bigskip
Multiplying by four, we get:
\begin{equation}
\label{n0n2}
4n_0n_2 \in
\left\{  a^2-(a-2h)^2 \ : \ 0\leq h \leq \left\lfloor \frac{a}{2} \right\rfloor \right\}
\end{equation}
and  
\begin{equation}
\label{n1n3}
4n_1n_3
\in 
\left\{  b^2-(b-2h)^2 \ : \ 0\leq h \leq \left\lfloor \frac{b}{2} \right\rfloor \right\}.
\end{equation}

Now note that $2n_0n_2+2n_1n_3$ is the number of differences in $\Delta X$ that are congruent to $2$ modulo $4$, 
which of course is also equal to $\frac{v}{4}-2j$.
Thus, from (\ref{n0n2}) and (\ref{n1n3}), there are  integers $h_1$, $h_2$ such that
$0\leq h_1 \leq \left\lfloor \frac{a}{2} \right\rfloor$, $0\leq h_2 \leq \left\lfloor \frac{b}{2} \right\rfloor$
and 
\begin{equation}
\label{square.eq}
a^2-(a-2h_1)^2+b^2-(b-2h_2)^2=\frac{v}{2}-4j.
\end{equation}
Using the facts that
\begin{align*} 
a+b&=k \quad \text{and}\\ 
ab &= \frac{v}{4}-i,
\end{align*} 
we have
\begin{align*}
a^2+b^2 &= (a+b)^2- 2ab \\
& =k^2-\frac{v}{2}+2i.
\end{align*}
Substituting this into (\ref{square.eq}), we have
\[k^2-\frac{v}{2}+2i -(a-2h_1)^2-(b-2h_2)^2=\frac{v}{2}-4j,\]
or \[ k^2-v +2i+4j = (a-2h_1)^2 + (b-2h_2)^2.\]
Since $v=k^2-k+2\ell$, we obtain 
\[ k - 2\ell +2i+4j = (a-2h_1)^2 + (b-2h_2)^2.\]
We conclude that $k-2\ell+2i+4j$ is a sum of two squares.

\bigskip

Suppose  $v\equiv 4 \bmod 8$.
As before, 
there are an even number, say $2i$,
of odd elements in $L(X)$, where $0\leq i\leq \ell-1$. 
However, $v/2 \equiv 2 \bmod 4$, so the number of elements $\equiv 2 \bmod 4$ that are not in $\Delta X$ is an odd number, say $2j+1$, 
where $0\leq j\leq \ell-1-i$.
 
Reasoning exactly as in the case where $v\equiv 0 \bmod 8$, we  find that $k-2\ell+4i$ is a perfect square
and that $k-2\ell+2i+4j+2$ is a sum of two squares.
\end{proof}

We now give an application of Theorem \ref{new3.5.thm}.

\begin{corollary}
\label{new3.5.cor}
Suppose that $k = n^2 -2\ell + 4$ where $\ell \geq 1$ and $n \geq \ell + 1$. Let $v = k^2 - k + 2\ell$.
\begin{enumerate}
\item If $v \equiv 0 \bmod 8$ and $k-2$ is not the sum of two squares, then a $(v,k)$-MGR does not exist.
\item If $v \equiv 4 \bmod 8$ and $k$ is not the sum of two squares, then a $(v,k)$-MGR does not exist.
\end{enumerate}
\end{corollary}

\begin{proof}
We note that $k - 2\ell + 4(\ell-1) = n^2$ is a perfect square. We claim there are no squares 
of the form $k-2\ell+4i$ where $0 \leq i \leq \ell-2$.
This is because the smallest such integer is 
\begin{align*} k - 2\ell &= n^2 - 4\ell+ 4\\
& \geq n^2 - 4(n-1)+ 4\\
& = n^2 -4n+8\\
& = (n-2)^2 +4.
\end{align*}
Since all these integers have the same parity as $n^2$ and they are not larger than $k - 2\ell + 4(\ell-1) = n^2$, the result follows.
Therefore $i=\ell-1$  is the only value in $[0,\ell-1]$ such that $k-2\ell+4i$ is a perfect square.

Now, in applying Theorem \ref{new3.5.thm}, we need to check that a certain condition holds for 
$0 \leq j \leq \ell-1 - i$. Since $i = \ell-1$, we only need to consider $j=0$. Theorem \ref{new3.5.thm} then states that 
a $(v,k)$-MGR does not exist if $v \equiv 0 \bmod 8$ and $k - 2\ell +2(\ell-1)  = k-2$ is not a sum of two squares; or if
$v \equiv 4 \bmod 8$ and $k - 2\ell +2(\ell-1) + 2  = k$ is not a sum of two squares. (It is not hard to verify that $v 
\equiv 0 \bmod 4$, so either $v \equiv 0 \bmod 8$ or $v \equiv 4 \bmod 8$.)
\end{proof}

We give some examples to illustrate results that can be obtained using Corollary \ref{new3.5.cor}.

\begin{example}
Suppose we take $n = 4t+2$ and $\ell=5$ in Corollary \ref{new3.5.cor}.  Then $v = k^2 - k + 10 \equiv 0 \bmod 8$.
Here we have \[k-2 = (4t+2)^2 -10+4-2 = 4(4t^2 + 4t-1).\] This integer is not the sum of two squares because 
$4t^2 + 4t-1 \equiv 3 \bmod 4$. Hence, no  $(k^2-k+10,k)$-MGR exists if $k = 4(2t+1)^2 - 6$. 
The first values of $k$ covered by this result are 
$k= 30, 94, 190, 318, 478, 670, 894, 1150, 1438, 1758$. 
\end{example}

\begin{example}
Suppose we take $n = 4t+2$ and $\ell=3$ in Corollary \ref{new3.5.cor}.  Then $v = k^2 - k + 6 \equiv 0 \bmod 8$.
Here we have \[k -2 = (4t+2)^2 - 6+4-2 = 16t^2 +16t.\] This integer is  the sum of two squares if and only if 
$t^2+t$ is the sum of two squares. Hence, no  $(k^2-k+10,k)$-MGR exists if $k = 4(2t+1)^2 - 2$ and
$t^2+t$ is not the sum of two squares. 
The first values of $k$ covered by this result are 
$k = 98, 194, 482, 674, 898, 1762, 2114, 2498, 2914, 3362$. 
\end{example}

\begin{example}
Suppose we take $n = 4t$ and $\ell=5$ in Corollary \ref{new3.5.cor}.  Then $v = k^2 - k + 10 \equiv  4 \bmod 8$.
Here we have \[k   = 16t^2 -6 = 2(8t^2 - 3).\] This integer is  the sum of two squares if and only if 
$8t^2 - 3$ is the sum of two squares. Hence, no  $(k^2-k+10,k)$-MGR exists if $k = (4t)^2 - 6$ and
$8t^2 - 3$ is not the sum of two squares. 
The first values of $k$ covered by this result are 
$k=138,  570, 1290, 2298, 2698, 3594, 5178, 6394, 7050, 9210$.
\end{example}

\section{Nonexistence Results for $(v,k,1)$-OOC}
\label{nonexistenceOOC.sec}

In this section, we prove nonexistence results for some optimal $(v,k,1)$-optical orthogonal codes of size $n >1$.
Note that we are investigating the cases where $v$ is even in this section.

\begin{lemma} 
\label{optimal2.lem}
Suppose $1 \leq \ell \leq \binom{k}{2}$ and $v = k(k-1)n+2\ell$. 
Then a  $(v,k,1)$-OOC of size $n$ is optimal.
\end{lemma}

\begin{proof}
For $v$ as given, we have 
\[ \left\lfloor \frac{v-1}{k(k-1)} \right\rfloor = n + \left\lfloor \frac{2\ell-1}{k(k-1)} \right\rfloor.\]
However, $2 \ell - 1 < k(k-1)$ because $\ell \leq \binom{k}{2}$, so 
\[ \left\lfloor \frac{v-1}{k(k-1)} \right\rfloor = n .\]
\end{proof}

Suppose $X=\{X_1,\dots,X_n\}$  is a $(v,k,1)$-optical orthogonal code. We define  $\Delta X$ and the leave, $L(X)$, in the obvious way:
\[ \Delta X = \bigcup_{i=1}^n \{ x - y \bmod v: x,y \in X_i, x \neq y\} \]
and 
\[ L(X) = \zed_v \setminus \Delta X .\]

The following lemma is a straightforward generalization of Lemma \ref{even-odd.lem}.

\begin{lemma}
\label{even-odd-O.lem}
Suppose $X$ is a $(v,k,1)$-optical orthogonal code where $v$ is even. Then
$\{0,v/2\} \subseteq L(X)$. If $v \equiv 0 \bmod 4$, then
 $|L_0(X)|$ and $|L_1(X)|$ are both even. If $v \equiv 2 \bmod 4$, then
$|L_0(X)|$ and $|L_1(X)|$ are both odd.
\end{lemma}

\begin{theorem}\label{CountingArgument}
Given $v=k(k-1)n+2\ell$ with $1 \leq \ell \leq \binom{k}{2}$, define the two sets 
\[S= \left\{ \left\lfloor \frac{v}{4} \right\rfloor - h \ : \ 0\leq h\leq \ell-1 \right\} .\]
and
\[T=\left\{h(k-h) \ : \ 0\leq h\leq \left\lfloor \frac{k}{2} \right\rfloor \right\}.\]
Then a necessary condition for the existence of an optimal $(v,k,1)$-OOC 
is that at least one element of $S$  is representable as a sum of $n$ integers of $T$.
\end{theorem}

\begin{proof}
Note than an optimal $(v,k,1)$-OOC will have size $n$, from Lemma \ref{optimal2.lem}.
Assume that $X=\{X_1,\dots,X_n\}$ is an (optimal) $(v,k,1)$-OOC.

From Lemma \ref{even-odd-O.lem}, we see that $v/2 \in L(X)$ and 
$|L_1(X)|$ has the same parity as $\frac{v}{2}$.
Also, as in the proof of Lemma \ref{counting2.thm}, $0 \leq |L_1(X)| \leq 2\ell-2$.
Thus, considering that the number of odd elements in 
$\zed_v$ is $v/2$, we see that the number of odd differences  
in $\bigcup_{i=1}^n \Delta X_i$ is twice an element of $S$.

Suppose that $X_i$ contains exactly $a_i$ even elements, so  $k-a_i$ is the number of  odd elements in $X_i$.
Then the number of odd elements in $\Delta X_i$ is $2a_i(k-a_i)$, that is, twice an element of $T$. It follows that
at least one element of $S$ is representable as a sum of $n$ integers belonging to $T$.
\end{proof}

Let us see some consequences of Theorem \ref{CountingArgument}. 
As a first example, we consider the cases where $k=3$.

\begin{corollary}
\label{k=3ooc}
An optimal  $(v,3,1)$-OOC does not exist if $v \equiv 14,20 \bmod 24$. 
\end{corollary}
\begin{proof}
When we take $k=3$ in Theorem \ref{CountingArgument}, we  have
$T = \{0,2\}$. Suppose $v$ is even and we write $v = 24t + 2w$, where $1 \leq w \leq 12$.  
We express $v$ in the form $v = 6n+2\ell$, where $1 \leq \ell \leq 3$, obtaining  
the values of $n$ and $\ell$ and the sets $S$ that are shown in Table \ref{k=3.tab}.

\begin{table}
\caption{Applications of  Theorem \ref{CountingArgument} when $k=3$}
\label{k=3.tab}
\[
\begin{array}{r|r|c|c}
 \multicolumn{1}{c|}{v}     & \multicolumn{1}{c|}{n}   &  \ell & S \\ \hline
24t+ 2  &  4t  & 1    & \{ 6t\} \\
24t+ 4  &  4t  & 2    & \{ 6t,6t+1\} \\
24t+ 6  &  4t  & 3    & \{ 6t-1,6t,6t+1\}\\
24t+ 8  &  4t+1  & 1  & \{ 6t+2\}\\
24t+ 10  &  4t+1  & 2 & \{ 6t+1,6t+2\}\\
24t+ 12  &  4t+1  & 3 & \{ 6t+1,6t+2,6t+3\}\\
24t+ 14  &  4t+2  & 1 & \{ 6t+3\}\\
24t+ 16  &  4t+2  & 2 & \{ 6t+3,6t+4\}\\
24t+ 18  &  4t+2  & 3 & \{ 6t+2,6t+3,6t+4\}\\
24t+ 20  &  4t+3  & 1 & \{ 6t+5\}\\
24t+ 22  &  4t+3 & 2 & \{ 6t+4,6t+5\}\\
24t+ 24  &  4t+3 & 3 & \{ 6t+4,6t+5,6t+6\}
\end{array}
\]
\end{table}
When $v \equiv 14,20 \bmod 24$, the set $S$ consists of a single element, which is an odd integer. Clearly it is not a sum of even integers, so we conclude from Theorem \ref{CountingArgument} that an optimal  $(v,3,1)$-OOC does not exist if $v \equiv 14,20 \bmod 24$. 
\end{proof}

\begin{remark}
It is well-known that an optimal  $(v,3,1)$-OOC exists if and only if $v \not\equiv 14,20 \bmod 24$ (e.g., see \cite{AB,BK} for discussion about this result).
\end{remark}

We adapt the  argument used in Corollary \ref{k=3ooc} to prove a generalization that works for 
odd integers $k \not\equiv 1 \bmod 8$.
First, we observe that, if $k$ is odd, then all the elements of $T$ are even. 
So we obviously get a contradiction in Theorem \ref{CountingArgument} if the set $S$ consists of a single odd integer. 
This happens if $\ell = 1$ (so $v = nk(k-1)+2$) and one of the following two conditions hold:
\begin{enumerate}
\item $nk(k-1) \equiv 2 \bmod 8$ ($v \equiv 0 \bmod 4$ in this case) or
\item $nk(k-1) \equiv 4 \bmod 8$ ($v \equiv 2 \bmod 4$ in this case).
\end{enumerate}

Since $k$ is odd, we have $k \equiv 1,3,5,7 \bmod 8$. We consider each case separately.

\begin{description}
\item [$\mathbf{k \equiv 1 \bmod 8}$:] \mbox{\quad} \vspace{.05in} \\ Here $k(k-1) \equiv 0 \bmod 8$, neither of 1.\ or 2.\ can hold.
\item [$\mathbf{k \equiv 3 \bmod 8}$:] \mbox{\quad} \vspace{.05in} \\ Here $k(k-1) \equiv 6 \bmod 8$. For 1., we obtain
$6n \equiv 2 \bmod 8$, so $n \equiv 3 \bmod 4$ and  $v = (4t+3)k(k-1) + 2$ for some integer $t$. It follows that
\[v\equiv 3k(k-1) + 2 \bmod 4k(k-1).\]
For 2., we obtain
$6n \equiv 4 \bmod 8$, so $n \equiv 2 \bmod 4$ and  $v = (4t+2)k(k-1) + 2$ for some integer $t$. It follows that
\[v\equiv 2k(k-1) + 2 \bmod 4k(k-1).\] 
\item [$\mathbf{k \equiv 5 \bmod 8}$:] \mbox{\quad} \vspace{.05in} \\ Here $k(k-1) \equiv 4 \bmod 8$. For 1., we obtain
$4n \equiv 2 \bmod 8$, which is impossible.
For 2., we obtain
$4n \equiv 4 \bmod 8$, so $n \equiv 1 \bmod 2$ and  $v = (2t+1)k(k-1) + 2$ for some integer $t$. It follows that
\[v\equiv k(k-1) + 2 \bmod 2k(k-1).\] 
\item [$\mathbf{k \equiv 7 \bmod 8}$:] \mbox{\quad} \vspace{.05in} \\ Here $k(k-1) \equiv 2 \bmod 8$. For 1., we obtain
$2n \equiv 2 \bmod 8$, so $n \equiv 1 \bmod 4$ and  $v = (4t+1)k(k-1) + 2$ for some integer $t$. It follows that
\[v\equiv k(k-1) + 2 \bmod 4k(k-1).\]
For 2., we obtain
$2n \equiv 4 \bmod 8$, so $n \equiv 2 \bmod 4$ and  $v = (4t+2)k(k-1) + 2$ for some integer $t$. It follows that
\[v\equiv 2k(k-1) + 2 \bmod 4k(k-1).\] 
\end{description}
Summarizing the above discussion, we have the following theorem.

\begin{theorem}
\label{4.3}
There does not exist an optimal $(v,k,1)$-OOC whenever one of the following conditions hold:
\begin{itemize}
\item $k \equiv 3 \bmod 8$ and $v\equiv 3k(k-1) + 2 \bmod 4k(k-1)$.
\item $k \equiv 3 \bmod 8$ and $v\equiv 2k(k-1) + 2 \bmod 4k(k-1)$.
\item $k \equiv 5 \bmod 8$ and $v\equiv k(k-1) + 2 \bmod 2k(k-1)$.
\item $k \equiv 7 \bmod 8$ and $v\equiv k(k-1) + 2 \bmod 4k(k-1)$.
\item $k \equiv 7 \bmod 8$ and $v\equiv 2k(k-1) + 2 \bmod 4k(k-1)$.
\end{itemize}
\end{theorem}

The following results are immediate corollaries of Theorem \ref{4.3}.
\begin{corollary}
An optimal $(v,3,1)$-OOC does not exist if
$v \equiv 14,20 \bmod 24$; an  optimal $(v,5,1)$-OOC does not exist if
$v \equiv 22 \bmod 40$; and an  optimal $(v,7,1)$-OOC does not exist if
$v \equiv 44,86 \bmod 168$.
\end{corollary}

\begin{example}
\label{62.exam}
As an example where Theorem \ref{CountingArgument} can be applied to an even value of $k$, consider the case of an
optimal $(62,6,1)$-OOC. Here we have $62 = 2\times 6 \times 5 + 2 \times 1$, so $n = 2$ and $\ell = 1$.
The set $S= \{15\}$ and $T= \{ 0,5,8,9\}$. It is impossible to express $15$ as the sum of two numbers from $T$, so we conclude
that an optimal $(62,6,1)$-OOC does not exist.
\end{example}

We now prove some general nonexistence results. 

\begin{theorem}
\label{twosquares.lem}
Suppose $1 \leq \ell \leq \binom{k}{2}$, and suppose an optimal $(2k(k-1) + 2\ell,k,1)$-OOC exists. Define the set $R$ as follows:
\[ R = \begin{cases}
\left\{ \left\lfloor \frac{k-\ell+1}{2} \right\rfloor + h \ : \ 0 \leq h \leq \ell -1 \right\} & \text{if $k$ is even}\\
\left\{ k-\ell+2h \  : \ 0 \leq h \leq \ell -1  \right\} & \text{if $k$ is odd and $\ell$ is even}\\
\left\{ k-\ell+2h+1 \  : \ 0 \leq h \leq \ell -1  \right\} & \text{if $k$ and   $\ell$ are both odd.}
\end{cases}
\]
Then
at least one integer in the set $R$ 
can be expressed as the sum of two squares.
\end{theorem}

\begin{proof}
First, suppose $k$ is even. Apply Theorem \ref{CountingArgument}. 
We have $v = 2k(k-1) + 2\ell$ and thus we have  
\[ S = \left\{ \frac{k(k-1)}{2} + \left\lfloor \frac{\ell}{2} \right\rfloor - h \ : 0 \leq h \leq \ell-1\ \right\} .\]
From Lemma \ref{numerical.lem}, we have 
\[T=\left\{ \left(\frac{k}{2}\right)^2-h^2 \ : \ 0\leq h\leq \frac{k}{2} \right\}.\]
From Theorem \ref{CountingArgument}, we have
\[ \frac{k(k-1)}{2} + \left\lfloor \frac{\ell}{2} \right\rfloor - h 
= \left(\frac{k}{2}\right)^2-i^2 + \left(\frac{k}{2}\right)^2-j^2\] 
for integers $h,i,j$ where $0 \leq h \leq \ell-1$ and $0 \leq i,j \leq k/2$.
Simplifying, we obtain
\[ \frac{k}{2} -  \left\lfloor \frac{\ell}{2} \right\rfloor + h = i^2 + j^2.\] 
The result follows by noting that
\[ \frac{k}{2} -  \left\lfloor \frac{\ell}{2} \right\rfloor = \left\lfloor \frac{k-\ell+1}{2} \right\rfloor\]
since $k$ is even.

\bigskip

Next, suppose $k$ is odd and $\ell$ is even.  Here $v \equiv 0 \bmod 4$. We again apply Theorem \ref{CountingArgument}. Here 
we have  
\[ S = \left\{ \frac{k(k-1) + \ell}{2}  - h \ : 0 \leq h \leq \ell-1\ \right\} \]
and, from Lemma \ref{numerical.lem}, we have 
\[T=\left\{ \left(\frac{k}{2}\right)^2-\left(\frac{k}{2} -h\right)^2 \ : \ 0\leq h\leq \frac{k-1}{2} \right\}.\]
From Theorem \ref{CountingArgument},  we get 
\[ \frac{k(k-1) + \ell}{2} - h 
= \left(\frac{k}{2}\right)^2- \left(\frac{k}{2} -i\right)^2  + \left(\frac{k}{2}\right)^2- \left(\frac{k}{2} -j\right)^2\] 
for integers $h,i,j$ where $0 \leq h \leq \ell-1$ and $0 \leq i,j \leq (k-1)/2$.
Simplifying, we have
\[
2k(k-1) + 2\ell - 4h = 2k^2- (k - 2i)^2 - (k - 2j)^2.\]
Therefore,
\[(k - 2i)^2 + (k - 2j)^2 = 2(k - \ell + 2h),\]
and the result follows.

\bigskip
The final case is when  $k$ and $\ell$ are both odd. The proof for this case is very similar to previous case.
\end{proof}

\begin{corollary}
\label{k/2.lem}
Suppose that $k$ has prime decomposition that contains a prime $p \equiv 3 \bmod 4$ raised to an odd power. Then an optimal $(2k(k-1) + 2,k,1)$-OOC does not exist.
\end{corollary}

\begin{proof}
Suppose an optimal $(2k(k-1) + 2,k,1)$-OOC exists.
Take $\ell = 1$ in Theorem \ref{twosquares.lem}; then $h=0$ in the definition of the set $R$.
It follows that, if $k$ is even, then $k/2$ is the sum of two squares; and if $k$ is odd, then $k$ is the sum of two squares.
The desired result then follows from Theorem \ref{sumofsquares.thm}.
\end{proof} 

\begin{remark}
The smallest applications of Corollary \ref{k/2.lem} are when $k=3$ and $k=6$.  We conclude that 
optimal $(14,3,1)$-OOC and optimal $(62,6,1)$-OOC do not exist. 
We note that Corollary \ref{k=3ooc} also shows that an optimal $(14,3,1)$-OOC does not exist.
Also, Example \ref{62.exam} proved the nonexistence of an optimal $(62,6,1)$-OOC using a slightly different argument.
The next values of $k$ covered by Corollary \ref{k/2.lem} are $k = 7,11,12,14,15,19,21,22,23,24$.
\end{remark}

Now we prove a nonexistence result that holds for arbitrarily large values of $\ell$.

\begin{theorem}
\label{infinite}
For any positive integer $\ell$, there are infinitely many even integers $k$ such 
that an optimal $(2k(k-1) + 2\ell,k,1)$-OOC
does not exist.
\end{theorem}

\begin{proof}
Using Lemma \ref{sequence.lem}, 
choose an even integer $k$ such that $\left\lfloor \frac{k-\ell+1}{2} \right\rfloor + h$ is not the  sum of two squares,
for $0 \leq h \leq \ell - 1$.  
Then apply  Theorem \ref{twosquares.lem}.
\end{proof}

We next prove the nonexistence of certain optimal $(3k(k-1)+2,k,1)$-OOC with $k$ even.

\begin{theorem} 
\label{ell=1.thm}
There does not exist an optimal $(3k(k-1)+2,k,1)$-OOC  if 
$k = (4^{a+1}(24c+7)+2)/3$ with $a ,c \geq 0$ or if 
$k = 4^{a+1}(8c+5)$ with $a , c \geq 0$.
\end{theorem}

\begin{proof}
Assume that $X$ is an optimal $(3k(k-1)+2,k,1)$-OOC
with $k$ even. We apply  Theorem \ref{CountingArgument} with $n=3$. 
Here, with the usual notation, we have
\[ S = \left\{ \frac{3k^2-3k+2}{4}  \right\} \] if $k \equiv 2 \bmod 4$, and
\[ S = \left\{ \frac{3k^2-3k}{4}  \right\} \] if $k \equiv 0 \bmod 4$.
Also, as in the proof of Theorem \ref{twosquares.lem}, we have 
\[T = \left\{ \left( \frac{k}{2} \right)^2 - h^2 \ : \ 0 \leq h \leq  \frac{k}{2} \right\}.\]
It follows that the unique element in the set $S$
must be a sum of three elements of $T$.

For $k \equiv 2 \bmod 4$, we have
\[ \frac{3k^2-3k+2}{4} = 3\left( \frac{k}{2} \right)^2 - ({h_1}^2+{h_2}^2+{h_3}^2)\] for
integers $h_1,h_2,h_3$. It follows that $(3k-2)/4$
is a sum of three squares, and hence $(3k-2)/4$ is not of the
form $4^a(8b+7)$ where $a,b \geq 0$.  Thus, if 
\begin{equation}
\label{4.9.1}
\frac{3k-2}{4} = 4^a(8b+7),
\end{equation} an optimal $(3k(k-1)+2,k,1)$-OOC does not exist.
(\ref{4.9.1}) holds if and only  if 
\[k = \frac{4^{a+1}(8b+7) + 2}{3}.\]
In order for $k$ to be an integer, $b$ must be divisible by $3$, say $b = 3c$. Therefore,  if
\[k = \frac{4^{a+1}(24c+7) + 2}{3},\]
where $a,c \geq 0$, an optimal $(3k(k-1)+2,k,1)$-OOC does not exist.

The case $k \equiv 0 \bmod 4$ is similar. Here, 
$3k/4$
must be a sum of three squares, and hence $3k/4$ is not of the
form $4^a(8b+7)$.
Therefore an  optimal $(3k(k-1)+2,k,1)$-OOC does not exist if
\[ k = \frac{4^{a+1}(8b+7)}{3}.\]
In order for $k$ to be an integer, we must have $b \equiv 1 \bmod 3$, say $b = 3c+1$.
Then $(8b+7)/3 = 8c+5$. We conclude that an optimal $(3k(k-1)+2,k,1)$-OOC does not exist if
\[ k = 4^{a+1}(8c+5),\] where $a,c \geq 0$.
\end{proof}

Finally, we prove the nonexistence of certain optimal $(3k(k-1)+4,k,1)$-OOC with $k$ even.

\begin{theorem} 
There does not exist an optimal $(3k(k-1)+4,k,1)$-OOC  if 
$k = (4^{a+3}(24c+23)-2)/3$ with $a ,c \geq 0$ or if 
$k = 4^{a+3}(8c+5)$ with $a , c \geq 0$.
\end{theorem}

\begin{proof}
We proceed as in the proof of Theorem \ref{ell=1.thm}, by applying  Theorem \ref{CountingArgument} with $n=3$. 
Assume that $X$ is an optimal $(3k(k-1)+4,k,1)$-OOC
with $k$ even. We have
\[ S = \left\{ \frac{3k^2-3k-2}{4} , \frac{3k^2-3k-2}{4} + 1 \right\} \] if $k \equiv 2 \bmod 4$, and
\[ S = \left\{ \frac{3k^2-3k}{4}, \frac{3k^2-3k}{4} + 1  \right\} \] if $k \equiv 0 \bmod 4$.
Also, 
\[T = \left\{ \left( \frac{k}{2} \right)^2 - h^2 \ : \ 0 \leq h \leq  \frac{k}{2} \right\}.\]
At least one element in the set $S$
must be a sum of three elements of $T$.

Suppose $k \equiv 2 \bmod 4$ and let $n = (3k+2)/4-1$. 
Proceeding as n the proof of Theorem \ref{ell=1.thm}, we see that one of $n$ or $n+1$ is the sum of three squares.
However, if $n+1 = 4^a(8b+7)$
where $a \geq 2$, then Lemma \ref{consec.lem}  implies that neither $n$ nor $n+1$ is the sum of three squares.
In this case, optimal $(3k(k-1)+4,k,1)$-OOC does not exist. 
This occurs when
\[ \frac{3k+2}{4} = 4^a(8b+7),\]
 with $a \geq 2$, or 
\[ k = \frac{4^{a+1}(8b+7)-2}{3}.\]
Since $k$ is an integer, $b \equiv 2 \bmod 3$, say $b = 3c+2$, and then
\[ k = \frac{4^{a+1}(24c+23)-2}{3},\] 
where $a \geq 2$. For $k$ of this form, an optimal $(3k(k-1)+4,k,1)$-OOC does not exist. 

Suppose $k \equiv 0 \bmod 4$ and let $n = 3k/4-1$. Here, by the same logic as above, 
an optimal $(3k(k-1)+4,k,1)$-OOC does not exist  when
\[ \frac{3k}{4} = 4^a(8b+7),\]
or \[ k = \frac{4^{a+1}(8b+7)}{3},\] where $a \geq 2$.
Here, $b \equiv 1 \bmod 3$, say $b = 3c+1$, and then
\[ k = 4^{a+1}(8c+5),\] where $a \geq 2$.
For $k$ of this form, an optimal $(3k(k-1)+4,k,1)$-OOC does not exist. 
\end{proof}

\section{Other Types of Designs}
\label{other.sec}

In this section, we obtain necessary conditions for the existence of certain cyclic Steiner 2-designs and relative difference families using the techniques we have developed. 

\subsection{Cyclic Steiner 2-designs}
\label{nonexistencecyclicBIBD.sec}

A \emph{Steiner 2-design of order $v$ and block-size $k$}, denoted as S$(2,k,v)$, consists of a set of $k$-subsets (called \emph{blocks}) of a $v$-set (whose elements are called \emph{points}) such that every pair of points occurs in a unique block.
An S$(2,k,v)$ is \emph{cyclic} if there is a cyclic permutation of the $v$ points that maps every block to a block.

It is well-known that a cyclic S$(2,k,v)$ 
exists only if $v\equiv 1 \text{ or } k \bmod k(k-1)$.
A cyclic S$(2,k,v)$ with $v\equiv 1 \bmod k(k-1)$ is  
equivalent to a $(v,k,1)$-OOC of size $n$; in this case the leave is $\{0\}$.
Further, a cyclic S$(2,k,v)$ with $v\equiv k \bmod k(k-1)$ is  
equivalent to a $(v,k,1)$-OOC of size $n$ whose leave 
is the subgroup of $\zed_v$ of order $k$.  

Assume that $X=\{X_1,\dots,X_n\}$ is an $(k(k-1)n+k,k,1)$-OOC of size $n$ that is obtained from a 
cyclic S$(2,k,k(k-1)n+k)$ with both $k$ and $n$ even. 
The leave  $L(X)$ has exactly $k/2$ odd elements and therefore the number of odd 
differences in $\bigcup_{i=1}^n \Delta X_i$ is $k(k-1)n/2$. 

Reasoning as in the proof of Theorem \ref{twosquares.lem}, we see that $k(k-1)n/4$ is the sum of $n$ integers
in the set 
\[T=\left\{ \left(\frac{k}{2}\right)^2-h^2 \ : \ 0\leq h\leq \frac{k}{2} \right\}.\]
Thus we have \[\frac{kn}{4}={h_1}^2+{h_2}^2+ \dots +{h_n}^2\] for a suitable $n$-tuple $(h_1,\dots,h_n)$ 
of nonnegative integers, each of which does not exceed  ${k}/{2}$. Using  Lagrange's Four-square Theorem 
(Theorem \ref{sumofsquares.thm}),
it is an easy exercise to see that such an $n$-tuple certainly exists for $n \geq 4$. 

However, if $n=2$, this is not always the case. Here we require
\[ \frac{k}{2}={h_1}^2+{h_2}^2\] for nonnegative integers $h_1, h_2 \leq k/2$.
As stated in Theorem \ref{sumofsquares.thm}, a positive integer 
can be written as the as a sum of two squares if and only if its prime decomposition contains no prime $p \equiv 3 \bmod 4$ raised to an odd power. So we obtain the following result.

\begin{theorem}
\label{twosquares}
If $k$ is an even integer whose prime decomposition contains a prime $p \equiv 3 \bmod 4$ raised to an odd power, then
there does not exists a cyclic S$(2,k,2k(k-1)+k)$.
\end{theorem}

We can apply Theorem \ref{twosquares} with $k = 6,12,14,22,24,28$, etc.

\bigskip

Now assume that $X=\{X_1,\dots,X_n\}$ is a $(k(k-1)n+k,k,1)$-OOC  that is obtained from a cyclic S$(2,k,k(k-1)n+k)$
with $k$ even and $n$ odd. 
Here all the elements of the leave of $X$ are odd, and hence 
all $(k(k-1)n+k)/2$ odd elements of $\zed_v$ have to appear in $\bigcup_{i=1}^n \Delta X_i$. Reasoning as above, 
we see that $\frac{k(k-1)n+k}{4}$ is the sum of $n$ integers in the set 
\[T=\left\{ \left(\frac{k}{2}\right)^2-h^2 \ : \ 0\leq h\leq \frac{k}{2} \right\},\] i.e.,
\[\frac{k(n-1)}{4}={h_1}^2+{h_2}^2+ \dots +{h_n}^2\] for a suitable $n$-tuple $(h_1,\dots,h_n)$ 
of nonnegative integers not exceeding ${k}/{2}$. Again, such a $n$-tuple exists
by Lagrange's Four-square Theorem if $n\geq 5$. 

But this is not always the case if $n=3$. 
Here we require
\[ \frac{k}{2}={h_1}^2+{h_2}^2+{h_3}^2\] for nonnegative integers $h_1, h_2, h_3 \leq k/2$.

Applying Legendre's Three-square Theorem (Theorem \ref{sumofsquares.thm}),
we  have the following result.

\begin{theorem}
If  $k = 2^a(8b+7)$ where $a$ and $b$ are nonnegative integers and $a$ is odd, 
then there does not exist a cyclic S$(2,k,3k(k-1)+k)$.
\end{theorem}

We can apply Theorem \ref{twosquares} with $k = 14, 46, 56, 62$, etc.

\subsection{Relative Difference Families}

When $G = \zed_v$ and the order of the subgroup $H$ is equal to $w$, a $(G,H,k,1)$-RDF is clearly a 
$(v,k,1)$-OOC whose leave is the subgroup of $\zed_v$ of order $w$.
In this case, some authors (e.g., \cite{yin}) speak of a \emph{$w$-regular} $(v,k,1)$-OOC.
Note that a $w$-regular $(v,k,1)$-OOC is optimal provided that $w\leq k(k-1)$.
Also, note that a $k$-regular $(v,k,1)$-OOC gives rise to a cyclic S$(2,k,v)$.

\begin{theorem}
\label{RDF.thm}
Let $G$ be a group with a subgroup $S$ of index $2$ and let $X$ be a $(G,H,k,\lambda)$-relative difference family of size $n$,
where  $|H| =w$.
If $H$ is contained in $S$, then $kn-\lambda w$ is a sum of $n$ squares. If $H$ is not contained in $S$, then $kn$ is a 
sum of $n$ squares.
\end{theorem}  

\begin{proof}
Let us say that an element of $G$ is \emph{even} or \emph{odd} according to whether it belongs to or does not belong to $S$, respectively.
Set $X=\{X_1,\dots,X_n\}$ and, for $i=1,\dots,n$, let $a_i$ and $b_i$ be the number of even and odd elements in $X_i$, respectively.
The number of odd elements in $\Delta X_i$ is $2a_ib_i$ (note that here we are treating 
$\Delta X_i$ and $\Delta X$ as  multisets since differences may be repeated). Also, by definition, the number 
of odd elements in  $\Delta X$ is $\lambda$ times the
number of all odd elements of $G\setminus H$. 

If $H$, $S$ are subgroups of a group $G$ with $|G:S|=2$,
then either $H\subseteq S$ or $|H\cap S|=|H|/2$.
Hence, we have \[\sum_{i=1}^n2a_ib_i= \frac{\lambda v}{2} \quad \text{or} \quad  \frac{\lambda (v-w)}{2},\] 
according to whether $H$ is contained or not contained in $S$. Thus we have: 
\[\sum_{i=1}^n4a_ib_i=\begin{cases}
\lambda v & \text{ if $H\subseteq S$}\\
\lambda(v-w) & \text{ if $H\not\subseteq S$.}
\end{cases}\]
Now, given that $a_i+b_i=k$, we have 
\[4a_ib_i=4a_i(k-a_i)=k^2-(k-2a_i)^2.\]
Replacing this in the above formula and taking account of (\ref{BaseBlocks}), we get
\[\sum_{i=1}^n(k-2a_i)^2=\begin{cases}
kn-\lambda w & \text{ if $H\subseteq S$}\\
 kn & \text{ if $H\not\subseteq S$.}
\end{cases}
\]
and the assertion follows.
\end{proof}

Theorem \ref{RDF.thm} is trivial for $n\geq 4$ in view of Theorem \ref{sumofsquares.thm}. On the other hand, it gives
some important information for $n=1,2,3$. We now discuss several consequences of Theorem \ref{sumofsquares.thm}.

First, we point out a connection with the Bose-Connor Theorem (Theorem  \ref{BC.thm}). Suppose we take $n=1$ in Theorem \ref{RDF.thm} and suppose $H \subseteq S$. Recall that $S$ is a subgroup of index $2$. 
Denote $|G| = v = uw$, where $|H| = w$.  Then Theorem \ref{RDF.thm} asserts that $k - \lambda w$ must be a perfect square. 
This result can also be obtained from Theorem  \ref{BC.thm}, as follows. The development of the
$(G,H,k,\lambda)$-relative difference family through the group $G$ yields a divisible design with $\lambda_1= 0$ and $\lambda_2 = \lambda$. Since $H$ and $S$ are  subgroups of $G$ and $H \subseteq S$, it must be the case that $w \mid \frac{v}{2}$, say 
$v/2 = tw$. Then $u = v/w = 2t$ is even.  Therefore statement 1 of Theorem  \ref{BC.thm} applies, and $k^2 - \lambda v$ is a square. However, $k(k-1) - \lambda (v-w)$ from (\ref{BaseBlocks}), so $k^2 - \lambda v = k - \lambda w$, so we obtain the same result.

In the special case of the preceding result where $w=1$, we see that $k - \lambda$ is a square. This also follows  from the Bruck-Ryser-Chowla Theorem (as we already discussed in Example \ref{BRC.exam} in the case where $G$ is cyclic).

If we take  $n=2$ and $w=1$, we see that, if $X$ is a $(v,k,\lambda)$-DF with two base blocks 
in a group with a subgroup of index 2, then $2k-\lambda$ is a sum of two squares (this result was first shown in \cite[Corollary 2.1]{MDV}).  Similarly, taking $n=3$ and $w=1$, we see that,  if $X$ is a $(v,k,\lambda)$-DF with three base blocks 
in a group with a subgroup of index 2, then $3k-\lambda$ is a sum of three squares (this result was first shown in \cite[Corollary 2.2]{MDV}).

Finally, $n\in\{2,3\}$ and $w=k\equiv 0 \bmod 2$, then a cyclic S$(2,k,k(k-1)n+k)$ exists only if $k$ is a sum of $n$ squares.
This is equivalent to results obtained in Section \ref{nonexistencecyclicBIBD.sec}.

\section{Summary}
\label{summary.sec}

We have proven a number of nonexistence results for infinite classes of modular Golomb rulers, optical orthogonal codes, cyclic Steiner systems and relative difference families. We note that very few results of this nature were previously known. Many of our new results are based on counting even and odd differences and then applying some classical results from number theory which establish which integers can be expressed as a sum of a two or three squares.

\section*{Acknowledgements} We would like to thank Dieter Jungnickel and Hugh Williams for helpful comments and pointers to the literature. We also thank Shannon Veitch for assistance with programming. Finally, we thank a referee for pointing out a mistake in Corollary \ref{BC.cor}.


\begin{thebibliography}{X}

\bibitem{AB}
R.J.R.\ Abel and M.\ Buratti.
Some progress on $(v,4,1)$ difference families and
optical orthogonal codes.
\emph{Journal of Combinatorial Theory  A} {\bf 106} (2004) 59--75.

\bibitem{BK}
C.M.\ Bird and D.A.\ Keedwell.
Design and application of
optical orthogonal codes---a survey.
\emph{Bulletin of the ICA} {\bf 11} (1994) 21--44.

\bibitem{bose}
R.C.\ Bose.
An affine analogue of Singer's theorem.
\emph{J.\ Indian Math.\ Soc.} {\bf 6} (1942), 1--15.



\bibitem{BC}
R.C.\ Bose and W.S.\ Connor.
Combinatorial properties of group divisible incomplete block designs.
\emph{Annals of Mathematical Statistics} {\bf 23} (1952), 367--383.


\bibitem{Buratti98}
M. Buratti. 
Recursive constructions for difference matrices 
and relative difference families.
\emph{Journal of Combinatorial Designs} {\bf 6} (1998), 165--182.


\bibitem{Buratti99}
M.\ Buratti.
Old and new designs via difference multisets and strong difference families.
\emph{Journal of Combinatorial Designs} {\bf 7} (1999), 406--425.

\bibitem{BS}
M.\ Buratti and D.R.\ Stinson.
On resolvable Golomb rulers, symmetric
configurations and progressive dinner parties.
Preprint.

\bibitem{chung}
F.R.K.\ Chung, J.A.\ Salehi and V.K.\ Wei.
Optical orthogonal codes: design, analysis and applications.
\emph{IEEE Transactions on Information Theory} {\bf 35} (1989), 595--604.


\bibitem{HCD}
C.J.\ Colbourn and J.H.\ Dinitz, Eds. 
\emph{Handbook of Combinatorial Designs, Second Edition}. Chapman \& Hall/CRC, 2007.

\bibitem{DFGMP}
A.A. Davydov, G.\ Faina, M.\ Giulietti, 
S.\ Marcugini and F.\ Pambianco. On constructions and parameters of symmetric
configurations $v_k$.
\emph{Designs Codes and Cryptography} {\bf 80} (2016), 125--147.


\bibitem{Dim} A.\ Dimitromanolakis. \emph{Analysis of the Golomb ruler and the Sidon set
problems, and determination of large, near-optimal Golomb rulers.} Masters
Thesis, Department of Electronic and Computer Engineering, Technical
University of Crete, June 2002.



\bibitem{Drakakis}
K.\ Drakakis. A review of the available construction methods for Golomb rulers.
\emph{Advances in Mathematics of Communications} {\bf 3} (2009), 235--250.


\bibitem{Dusart}    
P.\ Dusart.
Explicit estimates of some functions over primes.
\emph{Ramanujan Journal} {\bf 45} (2018),  227--251.

\bibitem{Gordon}
D.M.\ Gordon.
The prime power conjecture is true for $n<2,000,000$.
\emph{Electronic Journal of Combinatorics} {\bf 1} (1994), paper \# R6.

    

\bibitem{Jung}
D.\ Jungnickel. On automorphism groups of divisible designs. 
\emph{Canadian Journal of Mathematics} {\bf 34} (1982), 257--297.

\bibitem{Jung-survey}
D.\ Jungnickel. Difference sets. 
In 
\emph{Contemporary Design Theory: A Collection of Surveys}.
Wiley, 1992, pp.\ 241--324.

\bibitem{MDV}
L.\ Mart\'{i}nez,
D.\u{Z}. \DJ okovi\'{c} and 
A. Vera-L\'{o}pez.
Existence question for difference families and construction of some new families.
\emph{Journal of Combinatorial Designs} {\bf 12} (2004), 256--270.


\bibitem{Mordell}
L.J.\ Mordell. 
\emph{Diophantine Equations}.
Academic Press, 1969.

\bibitem{Nagura}
J.\ Nagura. On the interval containing at least one prime number. 
\emph{Proc.\ Japan Acad.} {\bf 28} (1952), 177--181.

\bibitem{Rosen}
K.H.\ Rosen. \emph{Elementary Number Theory and its Applications, Sixth Edition}.
Pearson, 2000.

\bibitem{ruzsa}
I.Z.\ Ruzsa.
Solving a linear equation in a set of integers I.
\emph{Acta Arithmetica} {\bf 53} (1993), 259--282.

\bibitem{yin}
J.\ Yin. Some combinatorial constructions for optical orthogonal codes. 
\emph{Discrete Math.} {\bf 185} (1998),  201--219.


\end{thebibliography}
\end{document}